\documentclass[11pt,a4paper,reqno]{amsart}

\usepackage[utf8]{inputenc}
\usepackage[a4paper, lmargin=0.12\paperwidth, rmargin=0.12\paperwidth, tmargin=0.11\paperheight, bmargin=0.11\paperheight]{geometry}
\usepackage{amsfonts,amsmath,amssymb,amsthm,dsfont,amsxtra,enumerate,multicol,mathtools,tkz-berge,hyperref,graphicx}

\theoremstyle{definition}
\newtheorem{theorem}{Theorem}[section]
\newtheorem{cor}{Corollary}[section]
\newtheorem{lemma}{Lemma}[section]
\newtheorem*{claim*}{Claim}
\newtheorem{prop}{Proposition}[section]
\newtheorem{example}{Example}

\newtheorem{definition}{Definition}[section]
\newtheorem{remark}{Remark}[section]
\newtheorem{notation}{Notation}[section]
\newtheorem{setup}{Set-up}[section]
\numberwithin{equation}{section}

\DeclareMathOperator{\pmd}{pmd}
\DeclareMathOperator{\rank}{Rank}

\title{On Positive Matching Decomposition Conjectures of Hypergraphs}

\author[Marie Amalore Nambi]{Marie Amalore Nambi}
\address{Sabanci University, Faculty of Engineering and Natural Sciences, Orta Mahalle, Tuzla, 34956, Istanbul, Turkey}
\email{amalore.p@gmail.com, amalore.pushparaj@sabanciuniv.edu}

\author[Neeraj Kumar]{Neeraj Kumar}
\address{Department of Mathematics, Indian Institute of Technology Hyderabad, Kandi, Sangareddy - 502285, INDIA}
\email{neeraj@math.iith.ac.in}

\subjclass[2020]{{Primary 05C70, 13F65, 13F70, 13C40}; Secondary {05C75, 05E40}} 

\keywords{LSS-ideal, complete intersection, matching, positive matching, alternate walk.}

\date{September 2024}

\begin{document}

\maketitle

\begin{abstract}
In this paper, we prove the conjectures of Gharakhloo and Welker  \cite[Conjecture 3.5 and Conjecture 3.6]{GW2023} that the positive matching decomposition number $(\pmd)$ of a $3$-uniform hypergraph is bounded from above by a polynomial of degree $2$ in terms of the number of vertices. Moreover, we derive a lower bound for $\pmd$ specifically for complete $3$-uniform hypergraphs. Additionally, we obtain an upper bound for $\pmd$ of $r$-uniform hypergraphs. As an application from an algebraic point of view, we obtain the radical, complete intersection, and prime properties of Lov\'{a}sz$-${S}aks$-${S}chrijver (LSS) ideals of $r$-uniform hypergraphs. For an $r$-uniform hypergraphs $H=(V,E)$ such that $\lvert e_i\cap e_j\rvert \leq 1$ for all $e_i,e_j \in E$, we give a characterization of positive matching in terms of strong alternate closed walks. For a specific class of hypergraphs, we classify the radical and complete intersection properties of LSS ideals. 
\end{abstract}

\section{Introduction}
Let $\mathbb{K}$ be a field, and $n\geq 1$ be an integer. Let $H=(V=[n],E)$ be a hypergraph such that $E$ is a clutter; that is, the sets in $E$ are pairwise incomparable with respect to inclusion. For an integer $d\geq 1$ and $e\subset [n]$ we consider the polynomial 
$$f_e^{(d)}=\sum_{j=1}^{d}\prod_{i\in e}x_{ij}$$
in the polynomial ring $S=\mathbb{K}[x_{ij}\mid i\in[n],j\in [d]]$. The ideal 
$$L_H^{\mathbb{K}}(d)=(f_e^{(d)} \mid e \in E) \subset S$$
is called the Lov\'{a}sz$-${S}aks$-${S}chrijver ideal (cf. \cite{HW2015}). We refer to it as LSS-ideal in short. Geometrically, when $H$ is a graph, the ideal $L_H^{\mathbb{K}}(d)$ coincides with the variety of orthogonal representations of the graph complementary to $H$ (cf. \cite{LSS89}). Lov\'{a}sz introduced the notion of orthogonal representations of graphs (cf. \cite{L79}). Many graph theoretical properties, like connectivity and chromatic number are related to orthogonal representations (cf. \cite{L2019}). The variety of orthogonal representations was initially studied by Lov\'{a}sz, {S}aks and {S}chrijver (cf. \cite{LSS89,LSS00}). For $d=1$, the ideal $L_H^{\mathbb{K}}(d)$ coincides with the edge ideal of $H$ (see for example, \cite{F2002, HV2008}). For $r=2$ and $d=2$, some algebraic properties of LSS-ideals, such as radical, prime, primary decomposition, complete intersection, and almost complete intersection, are studied in terms of combinatorial invariants of a graph (cf. \cite{HW2015, AK2021}).

A homogeneous ideal $I \subset S$ is said to be a complete intersection if $\mu(I) = ht(I)$, where $\mu(I)$ denotes the cardinality of a minimal homogeneous generating set of $I$. Characterizing complete intersection LSS ideals of a graph is crucial, as it provides essential insights into the characterization of radical, prime, and almost complete intersection properties of LSS ideals (cf. \cite{ ANC2023, CW2019, AK2021}). In \cite{CW2019}, the authors provide a characterisation for $L_H^{\mathbb{K}}(d)$ being radical, complete intersection and prime when $H$ is a tree and $d\geq2$. In \cite{ANC2023}, for $d\geq2$, the authors characterise almost complete intersection property of $L_H^{\mathbb{K}}(d)$ when $H$ is a forest, unicyclic, and bicyclic graphs. In Section \ref{sec.conj}, for $r=3$, we obtain a pattern on positive integer $d$ such that the ideal $L_H^{\mathbb{K}}(d)$ is a radical complete intersection. Our first main result is as follows:

\begin{theorem} \label{main1}
    Let $H$ be a $3$-uniform hypergraph. If $d \geq \frac{3}{2}n^2-\frac{15}{2}n+10$, then: 
    \begin{enumerate}
        \item $L_H^{\mathbb{K}}(d)$ is a radical complete intersection;
        \item $L_H^{\mathbb{K}}(d+1)$ is a prime.
    \end{enumerate}
\end{theorem}

In a similar fashion, for $r\geq 4$ and $d\geq 1$, we obtain condition on $d$ such that the ideal $L_H^{\mathbb{K}}(d)$ is a radical complete intersection. We prove Theorem \ref{main1} using a combinatorial invariant called the positive matching decomposition number of $H$, which was introduced by Conca and Welker (cf.  \cite{CW2019}).
 
We recall a subset $M \subseteq E$ is said to be \textit{matching} if $e \cap e'=\emptyset$ for all $e,e'\in M$ and $e \neq e'$. A \textit{positive matching} of hypergraph $H$ is a matching $M \subseteq E$  such that there exists a weight function $w: V \rightarrow \mathbb{Q}$ satisfying:
    \begin{equation} \label{pmd_condition}
      \sum_{i\in e} w(i)>0 \text{ if } e \in M, \hspace{2cm} \sum_{i\in e} w(i)<0 \text{ if } e \in E\setminus M. 
    \end{equation}
A \textit{positive matching decomposition} (or pm-decomposition) of $H$ is a partition $E = \cup_{i=1}^{p} M_i$ into pairwise disjoint subsets such that $M_i$ is a positive matching on $(V, E \setminus \cup_{j=1}^{i-1} M_j)$ for $i = 1,\ldots,p$. The $M_i$ are called the parts of the pm-decomposition. The smallest $p$ for which $H$ admits a pm-decomposition with $p$ parts will be denoted by $\pmd(H)$ (cf. \cite[Definition 5.1, Definition 5.3]{CW2019}).

The following implications establish a connection between $\pmd$ and the algebraic characteristics of LSS-ideals of hypergraphs:
\begin{equation} \label{eq.main}
    \pmd(H)\leq d \implies L_H^{\mathbb{K}}(d) \text{ is radical complete intersection } \implies L_H^{\mathbb{K}}(d+1) \text{ is prime.}
\end{equation}
The above implications are known for graphs due to Conca and Welker (see \cite[Theorem 1.3]{CW2019}). For hypergraphs, one may observe that the first implication follows from \cite[Lemma 5.5 and Proposition 2.4]{CW2019}. The second implication is proved in \cite[Theorem 1.2]{GW2023}. As a result, the concept of the $\pmd$ for hypergraphs assumes a crucial role as a noteworthy graph-theoretical invariant with versatile applications spanning the domains of algebra and geometry.

For an integer $r>1$, a hypergraph $H=(V,E)$ is said to $r$-uniform if $\lvert e\rvert = r$, for every $e\in E$. 
%In \cite{CW2019}, the authors obtained the upper and lower bounds for $\pmd$ in the case of $2$-uniform hypergraph, which is a graph. More precisely, 
In \cite[Theorem 5.4(1)]{CW2019}, the authors provided a linear upper bound of $\pmd(H)$ for a $2$-uniform hypergraphs. Gharakhloo and Welker established a matching decomposition for the $3$-uniform complete hypergraph $H=(V,E)$ with $n$ vertices and $\binom{n}{3}$ edges. Specifically, the authors proved that for every $3 \leq l_1 \leq 2n - 3$ and $5 \leq l_2 \leq 2n - 1$,
$E_{l_1,l_2} = \{\{a, b, c\} \in E \mid a < b < c, a + b = l_1, b + c = l_2\}$
is a matching and $E=\cup_{l_1,l_2}E_{l_1,l_2}$ (cf. \cite[Proposition 3.4]{GW2023}). Moreover, the authors proposed the following conjectures: 

\begin{theorem} \cite[Conjecture 3.5]{GW2023} \label{conjpmd}
    Let $H$ be a complete $3$-uniform hypergraph with $n$ vertices. Then $E_{l_1,l_2} = \{\{a, b, c\} \in E \mid a < b < c, a + b = l_1, b + c = l_2\}$ is a positive matching.
\end{theorem} 

\begin{theorem} \label{conj2}
\cite[Conjecture 3.6]{GW2023}
        Let $H=(V, E)$ be a $3$-uniform hypergraph with $n$ vertices. Then $\pmd(H) \leq \frac{3}{2}n^2-\frac{15}{2}n+10$.
\end{theorem}

Note that Theorem \ref{main1} follows from Theorem \ref{conj2} and  Theorem \ref{conj2} follows from Theorem \ref{conjpmd} and \cite[Proposition 3.4]{GW2023}. In Section \ref{sec.conj}, we provide an affirmative answer to Theorem \ref{conjpmd}. Moreover, we obtain a lower bound for $\pmd$ of a complete $3$-uniform hypergraph using a linear polynomial in the number of vertices, specifically $9n-35$ when $n \geq 5$ (see Theorem \ref{lowerboundpmd}). We demonstrate with an example that for a $3$-uniform hypergraph $H$ on $[7]$ vertices, $\pmd(H)$ is $28$ which is strictly less than $31$, the bound given in Theorem \ref{conj2}; refer to Example \ref{pmdstrict}.

%We demonstrate that the upper bound for $3$-uniform hypergraphs proposed in Conjecture \ref{conj2} is indeed strict when the number of vertices is greater than or equal to $7$; refer to Example \ref{pmdstrict}.

Notice that, in \cite[page-4042]{GW2023}, the authors mentioned that ``\textit{one can speculate that in general, for $r$-uniform hypergraphs $H$, the value of $\pmd(H)$ is bounded from above by a polynomial with a degree of $r-1$ in terms of the number $n$ of vertices}". In Subsection \ref{subsecpumd}, we derive a positive matching decomposition for the complete $r$-uniform hypergraphs for all $r \geq 4$. Notably, we obtain an upper bound value of $\pmd$ for $r$-uniform hypergraphs  (see Theorem \ref{thm.runibound}).

In \cite{FG2022}, the authors provide necessary and sufficient conditions for a matching of a graph to be a positive matching using alternating closed walks. Subsequently, the authors obtained $\pmd$ of complete multipartite graphs, bipartite graphs, cacti, and more. In \cite{GW2023}, the authors obtain the $\pmd$ of $r$-uniform hypertree, which is described as ``\textit{more restrictive compared to other hypertree definitions in the literature}" \cite[page-4038]{GW2023}. 

In Section \ref{secpmd}, we introduce the notion of strong alternating closed walk for hypergraphs. Let $H=(V,E)$ be an $r$-uniform hypergraph such that $\lvert e_i\cap e_j\rvert \leq 1$, for all $e_i,e_j \in E$ with $e_i\neq e_j$. Then, we prove that a matching $M$ in a hypergraph $H$ is positive if and only if the subgraph induced by $M$ does not contain any strong alternate closed walk (see Theorems \ref{thm.3-uni} and \ref{thm.r-uni}). We also introduce the notion of the good forest. A hypergraph $H$ is said to be a good forest if there exists a sequence of edges $e_1,\dots,e_m$ such that $\lvert V_{\{e_1,\ldots,e_i\}}\cap V_{e_{i+1}}\rvert \leq 1$, for all $1\leq i \leq m-1$, where $m=\lvert E\rvert$. The motivation for this condition is purely algebraic. Specifically, it facilitates obtaining a Borel-fixed ideal with the same Hilbert series as the LSS ideal of a good forest. We obtain the exact value of $\pmd$ for good forest, loose cycle hypergraphs and hypergraphs obtained from a hypergraph by adding pendant edges. Lastly, in Theorem \ref{Thm.cs} we prove that the LSS ideal of a good forest is the Cartwright-Sturmfels ideal (see Definition \ref{def.cs}). As an application, we obtain our final result as follows:

\begin{theorem} \label{thm.goodfor}
    Let $H$ be an $r$-uniform good forest. Then:
    \begin{enumerate}
        \item $L_H^{\mathbb{K}}(d)$ is radical for all $d$.
        \item $L_H^{\mathbb{K}}(d)$ is a complete intersection if and only if $d\geq \Delta(H)$.
        \item $L_H^{\mathbb{K}}(d)$ is prime if $d\geq \Delta(H)+1$.
    \end{enumerate}
\end{theorem}

\section{On Conjecture} \label{sec.conj}

In this section, we prove conjectures (Theorem \ref{conjpmd} and Theorem \ref{conj2}) proposed by Gharakhloo and Welker (cf. \cite{GW2023}). Namely, the $\pmd(H)$ for a $3$-uniform hypergraph $H$ is bounded from above by a quadratic function in the number of vertices. First, we recall relevant notations and known results.

\begin{remark} \cite[Lemma 5.2]{CW2019}  \label{Rem.pos}
    Let $H = (V, E)$ be a hypergraph, $M \subseteq E$ and $V_M = \cup_{A \in M}A$.    
    \begin{enumerate}[(a)]
        \item $M$ is a positive matching for $H$ if and only if $M$ is a positive matching for the induced hypergraph $(V_M, {A \in E \mid A \subseteq V_M})$.
        \item Assume $M$ is a positive matching on $H$ and $A \in E$ is such that $M_1 = M \cup \{A\}$  is a matching. Assume also that there is a vertex $a \in A$ such that $\{B \in E \mid B \subset V_{M_1}$ and $a \in B\} = \{A\}$. Then $M \cup \{A\}$ is a positive matching of $H$.
    \end{enumerate}
\end{remark}

\begin{setup}  \label{setup}
Let $H = (V, E)$ be the complete $3$-uniform hypergraph on $n$ vertices and with $\binom{n}{3}$
edges. Then for every $3 \leq l_1 \leq 2n - 3$ and $5 \leq l_2 \leq 2n - 1$, set $
E_{l_1,l_2} = \{\{a, b, c\} \in E \mid a < b < c, a+b = l_1, b+c = l_2\}$.
\end{setup}

\begin{remark} \label{form}
    Let $H$ be a complete $3$-uniform hypergraph as in Set-up \ref{setup}. Any element within $E_{l_1,l_2}$ takes on the following structure:
    $$\{l_1-l_2+\lambda,l_2-\lambda,\lambda\},$$
    where $3\leq \lambda \leq n$, $l_1-l_2+\lambda<l_2-\lambda<\lambda$, and $l_1-l_2+\lambda \geq 1$.
\end{remark}

\begin{remark} \cite[Proposition 3.4]{GW2023} \label{Rem.conj}
    Let $H$ be a complete $3$-uniform hypergraph as in Set-up \ref{setup}. Then $E_{l_1,l_2}$ is a  matching and $E = \cup_{l_1,l_2}E_{l_1,l_2}$. Moreover, the cardinality of the set $E_n=\{(l_1,l_2) \mid \text{ there exist } 1 \leq a < b < c \leq n, l_1 = a + b, l_2 = b + c\}$ is $\frac{3}{2}n^2-\frac{15}{2}n+10$.
\end{remark}

\begin{definition} \label{order}
    Let $(l_1,l_2)$ and $(l_1',l_2') \in \mathbb{N}^2$. We say $(l_1,l_2) < (l_1',l_2')$ if 
    \begin{enumerate}[(a)]
        \item $l_1 < l_1'$ or 
        \item $l_1 = l_1'$ and $l_2 < l_2'$.
    \end{enumerate}
\end{definition}
    
\begin{notation}
    Let $H$ be a $3$-uniform hypergraph as in Set-up \ref{setup}. Then the set $E_{l_1,l_2}^c$ denotes the set of all edges in an induced hypergraph $(V_{E_{l_1,l_2}},E\setminus \{\cup_{(l_1',l_2')<(l_1,l_2)}E_{l_1',l_2'} \cup E_{l_1,l_2}\})$.
\end{notation}

\begin{proof} [Proof of Theorem~\ref{conjpmd}]
    We show the pm-decomposition of $H$ by arranging the matching $E_{l_1,l_2}$ in the order defined in Definition \ref{order}. That is, we show a matching $E_{l_1,l_2}$ is a positive matching on $(V,E\setminus \cup_{(l_1',l_2')<(l_1,l_2)} E_{l_1',l_2'})$. 
    %We prove this by strong induction on $n$. For $n=3$, the statement holds. Now assume that the statement holds for all $j$, where $j < n$.  
    
    If $\lvert E_{l_1,l_2}\rvert = 1$ then it follows from Remark \ref{Rem.pos}(b)  that $E_{l_1,l_2}$ is positive matching on $(V,E\setminus \cup_{(l_1',l_2')<(l_1,l_2)}E_{l_1',l_2'})$. Let $a \in \mathbb{N}$ be an integer and we represent matching $E_{l_1,l_2}$ with $\lvert E_{l_1,l_2}\rvert >1$ in the following form:
    \begin{equation} \label{matrix}
        \begin{bmatrix}
            l_1-l_2+m-a & l_2-m+a & m-a\\
            l_1-l_2+m-a+1 & l_2-m+a-1 & m-a+1\\
            \vdots & \vdots & \vdots \\
            l_1-l_2+m-1 & l_2-m+1 & m-1\\
            l_1-l_2+m & l_2-m & m\\
        \end{bmatrix}
        =
        \begin{bmatrix}
            x_{11} & x_{12} & x_{13}\\
            x_{21} & x_{22} & x_{23}\\
            \vdots & \vdots & \vdots \\
            x_{a1} & x_{a2} & x_{a3}\\
            x_{a+1,1} & x_{a+1,2} & x_{a+1,3}\\
        \end{bmatrix}
    \end{equation}
    where $6\leq m \leq n$, $l_2-m+a < m-a$, $l_1-l_2+m-a \geq 1$. In Equation (\ref{matrix}) each row is an edge of $E_{l_1,l_2}$ (see Remark \ref{form}). We define a map $\rho: V(E_{l_1,l_2}) \rightarrow \mathbb{Q}$ as
    \begin{equation} \label{mapping}
        \begin{split}
            \rho(x_{11})&=t, \ \rho(x_{12})=-(\frac{t}{2} -1), \ \rho(x_{13})=-(\frac{t}{2} -1), \\
            \rho(x_{21})&=-(1+\rho(x_{12})+\rho(x_{13})), \\   \rho(x_{23})&=-(1+\rho(x_{11})+\rho(x_{12})), \\
            \rho(x_{22})&=1-(\rho(x_{21})+\rho(x_{23})), \\
            \rho(x_{i1})&=-(1+\rho(x_{i-1,2})+\rho(x_{i-2,2})), \\ 
            \rho(x_{i3})&=-(1+\rho(x_{i-1,1})+\rho(x_{i-1,2})), \\
            \rho(x_{i2})&=1-(\rho(x_{i1})+\rho(x_{i3})),
        \end{split}
    \end{equation}
    where $3\leq i \leq a+1$, $t \in \mathbb{N}$ such that  $\rho(x_{a+1,1}) > 0$ and $\rho(x_{a+1,2}) \leq 0$.

    To show the existence of such a $t$, first we claim that one can recursively obtain the following forms: $\rho(x_{i1})=t-\alpha_i$, $\rho(x_{i2})=-t/2+\beta_i$ and $\rho(x_{i3})=-t/2-2(i-1)$, where $\alpha_i$ and $\beta_i$ are some positive integers, for all $i>1$. The proof of the claim is by strong induction on $i$. For $i=2$, from Equation (\ref{mapping}) it follows that $\rho(x_{21})=t-3$, $\rho(x_{22})=-t/2+6$ and $\rho(x_{3})=-t/2-2$. Similarly, for $i=3$, it follows from Equation (\ref{mapping}) that $\rho(x_{31})=t-8$, $\rho(x_{32})=-t/2+13$ and $\rho(x_{33})=-t/2-4$. Now, assume the assertion holds for all $i<j$, for some $j\geq 4$. For statement $j$,  it follows from the induction hypothesis and Equation (\ref{mapping})  that $\rho(x_{j1})=t-(\beta_{j-1}+\beta_{j-2}+1)$. From Equation (\ref{eq.a3r}) and induction hypothesis it follows that $\rho(x_{j3})=-t/2-2(j-1)$. From Equation (\ref{mapping}) it follows that  $\rho(x_{j2})=-t/2+(\beta_{i-1}+\beta_{i-2}+2j)$. Hence, the claim is proved for $i=j$. 
    
    Thus, the conditions $\rho(x_{a+1,1}) > 0$ and $\rho(x_{a+1,2}) \leq 0$ hold if $t\geq 26$ when $a=1,2$, and  $t \geq 2(\beta_{i-1}+\beta_{i-2}+2a)$ when $a>2$.
    
    Let $E_{l_1,l_2}^c$ be the of set all edges in $(V_{E_{l_1,l_2}},E\setminus \{\cup_{(l_1',l_2')<(l_1,l_2)}E_{l_1',l_2'} \cup E_{l_1,l_2}\})$.
    Observe that $x_{11}<x_{21}< \cdots < x_{a+1,1}<x_{a+1,2}<x_{a2}< \cdots <x_{12}<x_{13} < \cdots <x_{a+1,3}$. From the observation, it follows that 
    \begin{equation*}
        E_{l_1,l_2}^c=
        \begin{cases}
            \{x_{11},x_{12}, \gamma\}, & \gamma=x_{23}, \ldots, x_{a+1,3}, \\
            \{x_{11}, \beta, \gamma\}, & \beta<\gamma, \text{ and } \beta,\gamma \in \{x_{13}, \ldots, x_{a+1,3}\}, \\
            %\{x_{21},x_{22}, \gamma\}, & \gamma=x_{33}, \ldots, x_{a+1,3} \\
            %\{x_{21}, \beta,\gamma\}, & \beta<\gamma, \beta,\gamma \in \{x_{12},x_{13}, \ldots, x_{a+1,3}\} \\
            \{x_{i1},x_{i2}, \gamma\}, & 2 \leq i \leq a, \text{ and } \gamma=x_{i+1,3}, \ldots, x_{a+1,3}, \\
            \{x_{i1}, \beta,\gamma\}, & 2 \leq i \leq a+1, \beta<\gamma, \text{ and } \beta,\gamma \in \{x_{i-1,2},\ldots,x_{12},x_{13}, \ldots, x_{a+1,3}\}, \\
            \{\alpha, \beta,\gamma\}, &  \alpha<\beta<\gamma, \text{ and } \alpha,\beta,\gamma \in \{x_{a+1,2},\ldots,x_{12},x_{13}, \ldots, x_{a+1,3}\}.
        \end{cases}
    \end{equation*}
    Clearly, these are all the edges in $E\setminus \{\cup_{(l_1',l_2')<(l_1,l_2)}E_{l_1',l_2'} \cup E_{l_1,l_2}\}$ induced by $V_{E_{l_1,l_2}}$.

    \vspace{2mm}

    From Remark \ref{Rem.pos}(a), it follows that to complete the proof, it is enough to prove the following two claims: 
    \begin{enumerate}
        \item \label{claim1} If $e\in E_{l_1,l_2}$ then $\sum_{i\in e}\rho(i)>0$;
        \item \label{claim2}  If $e'\in E_{l_1,l_2}^c$ then $\sum_{i\in e'}\rho(i)<0$.
    \end{enumerate}
    Claim (\ref{claim1}): From the map $\rho$ $(\ref{mapping})$ it is clear that $\rho(x_{11})+\rho(x_{12})+\rho(x_{13})=2$ and $\rho(x_{i1})+\rho(x_{i2})+\rho(x_{i3})=1$, for all $i=2,\ldots,a+1$. 

    \vspace{1mm}
    
    Before proving Claim (\ref{claim2}) we show that \begin{multline} \label{ineq}
         \rho(x_{11})> \rho(x_{21})>  \ldots >  \rho(x_{a+1,1})> \\
         \rho(x_{a+1,2})> \rho(x_{a2})>  \ldots > \rho(x_{12})= \rho(x_{13})>  \ldots > \rho(x_{a+1,3}).
    \end{multline}
It is clear that $\rho(x_{12})$ and $\rho(x_{13})$ are equal from the map $\rho$. For all $i=a+1,a,\ldots,3$, one has 
    \begin{equation} \label{eq.a3r}
        \begin{split}
            \rho(x_{i3})&=-1-\rho(x_{i-1,1})+\rho(x_{i-1,2})) \\
            &= -1-\rho(x_{i-1,1})-1+\rho(x_{i-1,1})+\rho(x_{i-1,3}) \\
            &<\rho(x_{i-1,3}).\\
        \end{split}
    \end{equation}
    \text{We have, }
    \begin{equation*} 
        \begin{split}
            \rho(x_{23})&=-1-\rho(x_{11})+\rho(x_{12})=-1-t+\rho(x_{12}) \\
            &<\rho(x_{12})=\rho(x_{13}).
        \end{split}
    \end{equation*}
    From Equation (\ref{mapping}) it follows that 
    \begin{equation*} \label{eq.2r}
    \begin{split}
       \rho(x_{11}) & > -(\rho(x_{12})+\rho(x_{13})) > \rho(x_{21}).\\
        \rho(x_{21}) & > -(\rho(x_{22})+\rho(x_{23})) \\
        & = -\rho(x_{22})-\rho(x_{12})+1+\rho(x_{11})+2\rho(x_{12})) \\
        &= -\rho(x_{22})-\rho(x_{12})+3 >  \rho(x_{31}).\\
    \end{split}
    \end{equation*}
For all $i=4,5, \ldots, a+1$,  one has
    \begin{equation} \label{eq.r1}
    \begin{split}
        \rho(x_{i1})&=-1-\rho(x_{i-1,2})-\rho(x_{i-2,2}) \\
        &= -3+\rho(x_{i-1,1})+\rho(x_{i-1,3})+\rho(x_{i-2,1})+\rho(x_{i-2,3}) \\
        &=\rho(x_{i-1,1}) -4 +\rho(x_{i-2,3})-\rho(x_{i-2,2}). \\
    \end{split}
    \end{equation}
From Equation  (\ref{mapping}) it follows that  $-4 +\rho(x_{2,3})-\rho(x_{2,2}) < 0$. 
Therefore from Equation \ref{eq.r1} we have, $\rho(x_{41})<\rho(x_{31})$. 

Consider the Fibonacci sequence $f_1=1, f_2=2$ and $f_n=f_{n-2}+f_{n-1}$ for all $n>2$. From Equation  (\ref{mapping}) it follows that  $\rho(x_{i-2,3})-\rho(x_{i-2,2}) < -f_{i-2}\rho(x_{11})-f_{i-1}\rho(x_{12})-f_{i-1}\rho(x_{13})<0$, for all $i=5,6,\ldots,a+1$. Thus from Equation (\ref{eq.r1}) we have, $\rho(x_{i1})<\rho(x_{i-1,1})$ for all $i=5,6,\ldots,a+1$.

Since $\rho(x_{a+1,1}) > 0$ and $\rho(x_{a+1,2}) \leq 0$, one has $\rho(x_{a+1,1})>\rho(x_{a+1,2})$. For all $i= a+1,\ldots,3$, from Equation (\ref{mapping}) we have,
    \begin{equation*}
        \begin{split}
            \rho(x_{i2})&=1-(\rho(x_{i1})+\rho(x_{i3})) \\
            &= 1+1+\rho(x_{i-1,2})+\rho(x_{i-2,2})+1+\rho(x_{i-1,1})+\rho(x_{i-1,2}) \\
            &=3+\rho(x_{i-1,2})+O(i),  \text{ where } O(i)=\rho(x_{i-1,2})+\rho(x_{i-2,2})+\rho(x_{i-1,1}).\\
         \end{split}
    \end{equation*}       
Since $\rho(x_{i-1,2})+\rho(x_{i-2,2}) = - \rho(x_{i1}) -1$  and $\rho(x_{i1}) < \rho(x_{i-1,1})$,  it follows that $O(i)=\rho(x_{i-1,2})+\rho(x_{i-2,2})+\rho(x_{i-1,1}) \geq 0$. Therefore, we get $\rho(x_{i2})>\rho(x_{i-1,2})$. Also, we have
    \begin{equation*}
        \begin{split}
            \rho(x_{22})&=1-(\rho(x_{21})+\rho(x_{23}))\\
            &= 1+1+\rho(x_{12})+\rho(x_{13})+1+\rho(x_{11})+\rho(x_{12})=5+\rho(x_{12}) \\
            &>\rho(x_{12}).\\
        \end{split}
    \end{equation*}
    Hence, Equation (\ref{ineq}) holds.

\vspace{1mm}
      
    \noindent Claim (\ref{claim2}): First we show that for edges $\{x_{11},x_{12}, \gamma\} \in E_{l_1,l_2}^c$, $\rho(x_{11})+\rho(x_{12})+ \rho(\gamma)<0$, where $\gamma\in \{x_{23}, \ldots, x_{a+1,3}\}$. From the map $\rho$ it follows that $\rho(x_{11})+\rho(x_{12})+ \rho(x_{23})=-1$. From Equation (\ref{ineq}) it follows that $\rho(x_{23})>\rho(j)$, for all $j\in \{x_{33}, \ldots, x_{a+1,3}\}$. Thus $\rho(x_{11})+\rho(x_{12})+\rho(j)<0$, as desired. Similarly, for every edge $\{x_{i1},x_{i2}, \gamma\} \in E_{l_1,l_2}^c$, one has $\rho(x_{i1})+\rho(x_{i2})+ \rho(\gamma)<0$, where $2 \leq i \leq a, \text{ and } \gamma\in \{x_{i+1,3}, \ldots, x_{a+1,3}\}$. From Equations (\ref{ineq}) and (\ref{mapping}) one has $\rho(x_{a+1,2}) \leq 0$, and $\rho(x_{a+1,2}) > \rho(j)$, for all $j\in \{x_{a,2},\ldots,x_{12},x_{13}, \ldots, x_{a+1,3}\}$. Therefore it is clear that for every edge  $\{\alpha, \beta,\gamma\} \in E_{l_1,l_2}^c$, one has  $\rho(\alpha)+ \rho(\beta)+\rho(\gamma)<0$, where $\alpha<\beta<\gamma, \text{ and } \alpha,\beta,\gamma \in \{x_{a+1,2},\ldots,x_{12},x_{13}, \ldots, x_{a+1,3}\}$.

    Next, we show that for edges  $\{x_{11}, \beta, \gamma\} \in E_{l_1,l_2}^c$, $\rho(x_{11})+ \rho(\beta)+ \rho(\gamma) <0$, for all  $\beta<\gamma, \text{ and } \beta,\gamma \in \{x_{13}, \ldots, x_{a+1,3}\}$. From Equation (\ref{mapping}) it follows that $\rho(x_{11})+\rho(x_{13})+\rho(x_{23})=-1$. Then from Equation (\ref{ineq}) one has $\rho(x_{13})+\rho(x_{23})\geq \rho(i)+\rho(j)$, for all $i<j$ and $i,j\in \{x_{13}, \ldots, x_{a+1,3}\}$. Hence we have $\rho(x_{11})+ \rho(\beta)+ \rho(\gamma) <0$ as desired. Similarly, for every edge  $\{x_{i1}, \beta,\gamma\} \in  E_{l_1,l_2}^c$, one has $\rho(x_{i1})+ \rho(\beta)+ \rho(\gamma) <0$, for all  $2 \leq i \leq a+1$, $\beta<\gamma, \text{ and } \beta,\gamma \in \{x_{i-1,2},\ldots,x_{12},x_{13}, \ldots, x_{a+1,3}\}$.  Thus $E_{l_1,l_2}$ is a positive matching on $(V,E\setminus \cup_{(l_1',l_2')<(l_1,l_2)} E_{l_1',l_2'})$ as desired.
\end{proof}

\begin{proof} [Proof of Theorem~\ref{conj2}]
    It follows from Theorem \ref{conjpmd} and Remark \ref{Rem.conj}.
\end{proof}

\begin{proof} [Proof of Theorem~\ref{main1}]
    It follows from Theorem \ref{conj2}.
\end{proof}

LSS-ideals can be geometrically viewed as follows (cf. \cite[Proposition 9.9]{CW2019}). ``\textit{Let $\mathbb{K}$ be an algebraically closed field and $H$ be an $r$-uniform hypergraph. Consider the map $\phi:(\mathbb{K}^n)^d \rightarrow \underbrace{\mathbb{K}^n \otimes \cdots \otimes \mathbb{K}^n}_r$ by 
$$(v_1,\ldots,v_d) \longmapsto \sum_{j=1}^{d}\underbrace{v_j\otimes\cdots \otimes v_j}_r=\sum_{j=1}^{d}\sum_{1\leq i_1,\ldots,i_r\leq n}(v_j)_{i_1}\cdots(v_j)_{i_r}e_{i_1}\otimes\cdots \otimes e_{i_r}$$
where $\{e_1,\ldots,e_n\}$ is the standard basis of $\mathbb{K}^n$. The Zariski closure of the image of $\phi$ is the variety $S_{n,r}^d$ of symmetric tensors of (symmetric) rank $\leq d$.  Let $\mathcal{V}(L_H^{\mathbb{K}}(d))$ be the vanishing locus of the ideal $L_H^{\mathbb{K}}(d)$. Then the restriction of the map $\phi$ to $\mathcal{V}(L_H^{\mathbb{K}}(d))$ is a parametrization of coordinate section of $S_{n,r}^d$ with $0$ coefficient at $e_{i_1}\otimes\cdots \otimes e_{i_r}$ for $\{i_1,\ldots, i_k\} \in E$. In particular, the Zariski-closure of the image of the restriction is irreducible if $L_H^{\mathbb{K}}(d)$ is prime. Hence, establishing the primality of $L_H^{\mathbb{K}}(d)$ serves as a valuable method for deducing the irreducibility of coordinate sections of $S_{n,r}^d$." }

\begin{remark}
    Let $H$ be an $r$-uniform hypergraph on $[n]$. If $d\geq \pmd(H)+1$, then the coordinate sections of the variety  $S_{n,r}^d$ with respect to $H$ are irreducible.
\end{remark}

As an application to Theorem \ref{main1} and using the known maximum bound for the generic rank of symmetric tensors of order $r$ on $\mathbb{K}^n$, which is given by $\frac{1}{n}\lceil \binom{n+r-1}{n-1}\rceil$ (see [\cite{BT2015}, Theorem 1.2] and [\cite{BO2008}, Section 3.1]),
%(see \cite[ Corollary 5.2]{L2010})
 we obtain the following result.

\begin{cor}
    Let $H$ be a $3$-uniform hypergraph on $[n]$. Then for $\frac{3}{2}n^2-\frac{15}{2}n+11 \leq d \leq \frac{1}{n}\lceil \binom{n+2}{3}\rceil -1$, every coordinate sections of the variety $S_{n,3}^d$ with respect to $H$ are irreducible.
\end{cor}

%\begin{remark}  \label{rem.sn3}
 %  From Conjecture \ref{conj2} and \cite[ Corollary 5.2]{L2010} it follows that  for $\frac{3}{2}n^2-\frac{15}{2}n+10 \leq d \leq \binom{n+2}{3}-n+1$, every coordinate sections of the variety $S_{n,3}^d$ is irreducible.
%\end{remark}

\subsection{Lower bound for pmd} 
In this subsection, we establish a lower bound for $\pmd$ of a complete $3$-uniform hypergraph using a linear polynomial in terms of the number of vertices. Additionally, we explore the equality of the $\pmd$ values for few complete $3$-uniform hypergraphs.

\begin{theorem} \label{lowerboundpmd}
    Let $H_n$ be a complete $3$-uniform hypergraph on $[n]$. Then $9n -35 \leq \pmd(H_n) \leq \frac{3}{2}n^2-\frac{15}{2}n+10$, for all $n \geq 5$.
\end{theorem}
\begin{proof}
    The upper bound follows from the Theorem \ref{conj2}.
    We prove the lower bound by induction on $n$. For $n=5$, the statement follows from the minimal matching decomposition, which is equal to the number of edges in $H_5$, that is $10$, since each matching has exactly one element.
    
    Suppose $\pmd(H_{n-1}) \geq 9(n-1) -35$. Consider $M_1, \ldots, M_p$ to be a positive matching decomposition of $H_n$ with $1 = \lvert M_1 \rvert \leq \cdots \leq \lvert M_p \rvert$. Let $r$ be the first integer such that $\lvert M_r \rvert > 1$. Therefore, $M_r$ contains at least two edges. First, we fix any two edges in $M_r$, say $e'_1$ and $e'_2$. Then, for a fixed vertex $u$ in $e'_1$, there are $9$ edges that have $u$ as a vertex in the induced subgraph $H[V_{\{e'_1,e'_2\}}]$. Let $e_1, \ldots, e_9$ be edges with a common vertex $u$ in the induced subgraph on $V(e'_1,e'_2)$. From Theorem \ref{pos.M_2} it follows that if $e_1, \ldots, e_9 \notin \cup_{j=1}^{r-1}M_j$ then $M_r$ is  not positive. Then it follows that 
 $$\pmd(H_n) = 9+ \pmd(H_n-\{e_1, \ldots, e_9\}) \geq 9+ \pmd(H_n\setminus u) = 9+ 9(n-1) -35,$$
as desired.
\end{proof}

In the following example, we produce a positive matching decomposition of a complete $3$-uniform hypergraph $H$, in which it attains the lower bound, i.e. $9n-35$, and strictly less than the upper bound, i.e. $\frac{3}{2}n^2-\frac{15}{2}n+10$, when $n \geq 5$.

\begin{example} \label{pmdstrict}
    Let $H$ be a complete $3$-uniform hypergraph on $[7]$.  Let $M_{22}=\{\{1,4,5\},\{2,3,6\}\}$, $M_{23}=\{\{1,4,6\},\{2,3,7\}\}$, $M_{24}=\{\{1,3,7\},\{2,4,5\}\}$, $M_{25}=\{\{1,5,6\},\{2,4,7\}\}$, $M_{26}=\{\{2,5,6\},\{3,4,7\}\}$, $M_{27}=\{\{1,2,3\},\{5,6,7\}\}$, and $M_{28}=\{\{1,6,7\},\{3,4,5\}\}$ be a matching of $E(H)$. Let $M_1, \ldots, M_{21}$ be distinct singleton edges from the edge set $E(G)\setminus \{M_{22},\dots,M_{28}\}$, taken in any order. We claim that $M_1, \ldots, M_{28}$ be the positive matching decomposition of $H$. Clearly, $M_i$ is positive matching on $(V(H),E(H)\setminus \cup_{j=1}^{i-1}M_j)$ for $i=1, \ldots, 21$. From Theorem \ref{pos.M_2} it follows that $M_{k}$ is positive matching on $(V(H),E(H)\setminus \cup_{j=1}^{k-1}M_j)$, where $k=22,\ldots,28$. 
\end{example}

In the following Table \ref{tab:pmd}, we compare the values of $\pmd$ of $3$-uniform hypergraph $H$ with the number of vertices of $H$. 

\begin{table}[ht]
    \centering
    \begin{tabular}{|c|c|c|}
    \hline
        Number of vertices of $H$ & $\pmd(H)$ & $\frac{3}{2}n^2-\frac{15}{2}n+10$ \\ \hline
        $3$ & $1$ & $1$ \\  \hline
        $4$ & $4$ & $4$ \\ \hline
        $5$ & $10$ & $10$ \\ \hline
        $6$ & $19$ & $19$ \\ \hline
        $7$ & $28$ & $31$\\ \hline
        $8$ & $37 \leq \pmd(H) \leq 38$ & $46$ \\ \hline
    \end{tabular}
    \vspace{5mm}
    \caption{Number of vertices of $H$ \textit{vs} $\pmd(H)$ \textit{vs} upper bound}
    \label{tab:pmd}
\end{table}

The equality of $\pmd(H)$ for a $3$-uniform hypergraph follows from the observation that each matching has exactly one element when the number of vertices is less than or equal to $5$. Additionally, in the case where the number of vertices is $6$, there is only one matching with a cardinality of two. This characteristic qualifies it as a minimal matching decomposition.

For a $3$-uniform hypergraph $H$ with $7$ vertices, the value of $\pmd$ is determined by applying Theorem \ref{lowerboundpmd} and Example \ref{pmdstrict}.

\subsection{\textit{r}-Uniform complete hypergraphs} \label{subsecpumd}
In this subsection, we derive an upper bound for $\pmd$ of complete $r$-uniform hypergraphs. This extension naturally encompasses the scenario when $r=3$. Subsequently, we endeavor to apply a parallel approach to demonstrate that $\pmd(H)$ is bounded above by a polynomial of degree $r-1$ in terms of the number of vertices when $H$ is an $r$-uniform hypergraph.

\begin{prop} \label{prop.r-uni}
    Let $H = (V, E)$ be the complete $r$-uniform hypergraph on $n$ vertices and with $\binom{n}{r}$ edges. Then for every $3 \leq l_1 \leq 2n-2r+3$, $5 \leq l_2 \leq 2n-2r+5, \ldots$, $2r-1 \leq l_{r-1} \leq 2n-1$, the set $E_{l_1,\ldots,l_{r-1}} = \{\{a_1,\ldots, a_{r}\} \in E \mid 1 \leq a_1 < \cdots < a_{r} \leq n, a_1+a_2 = l_1, a_2+a_3 = l_2,\ldots,a_{r-1}+a_{r}=l_{r-1}\}$ is a matching.
\end{prop}
\begin{proof}
    Let $e = \{x_1 < x_2 < \cdots < x_r\}$, $e' = \{y_1 < \cdots < y_r\} \in E_{l_1,\ldots,l_{r-1}}$ for some $3 \leq l_1 \leq 2n-2r+3$, $5 \leq l_2 \leq 2n-2r+5, \ldots$, $2r-1 \leq l_{r-1} \leq 2n-1$. Assume $e \neq e'$ and $e \cap e' \neq \emptyset$.
    
    \textbf{Case I.} If $x_1=y_1$ then since $x_1+x_2=l_1=y_1+y_2$ we have $x_2=y_2$. By repeating a similar argument, we get $x_i=y_i$ for all $i=3,\ldots,r$. Thus we have $e=e'$ a contradiction. The remaining cases $x_i=y_i$ can be deduced in a similar manner. 
    
    \textbf{Case II.} If $x_i=y_{i+1}$ or $x_{i+1}=y_i$, where $i=1,\ldots,r-1$, then from condition $x_i+x_{i+1}=l_i=y_i+y_{i+1}$ yield  a contradiction to the order of the elements in $e$ and $e'$. 
    
    \textbf{Case III.} If $x_i=y_j$ such that $i<j$ and $j-i\geq 2$, where $i=1,\ldots,r-2$, then $y_1<\cdots<y_j=x_i<x_{i+1}$ contradicts $x_i+x_{i
    +1}=y_i+y_{i+1}=l_i$.
    
    Hence, we have $e\cap e'=\emptyset$ as desired.
\end{proof}

\begin{prop} \label{prop.runibound}
    Let $H=(V,E)$ be the complete $r$-uniform hypergraph on $n$ vertices. Then the cardinality of the set $E_n=\{(l_1,\ldots,l_{r-1})\mid \text{ there exists  } 1 \leq a_1 < \cdots < a_{r} \leq n, a_1+a_2 = l_1, a_2+a_3 = l_2,\ldots,a_{r-1}+a_{r}=l_{r-1}\}$ is bounded from above by a polynomial of degree $r-1$ in terms of the number of vertices, that is $\lvert E_n \rvert \leq n^{r-1}+O(r-2)$, where $O(r-2)$ denotes a polynomial of degree less than or equal to $r-2$ in terms of $n$.
\end{prop}
\begin{proof}
    It is obvious that $3 \leq l_1<\cdots<l_{r-1} \leq 2n-1$. Thus the number of ways of choosing $r-1$ elements in increasing order from $[2n-3]$, that is, $\binom{2n-3}{r-1}$ is strictly greater than $\lvert E_n \rvert$. Hence, it can be deduced that  $\lvert E_n \rvert \leq n^{r-1}+O(r-2)$.
\end{proof}

\begin{theorem} \label{thm.runipmd}
    Let $H =(V,E)$ be an $r$-uniform hypergraph as in Proposition \ref{prop.r-uni}. Then $E_{l_1,\ldots,l_{r-1}} = \{\{a_1,\ldots, a_{r}\} \in E \mid a_1 < \cdots < a_{r}, a_1+a_2 = l_1, a_2+a_3 = l_2,\ldots,a_{r-1}+a_{r}=l_{r-1}\}$ is a positive matching.
\end{theorem}
\begin{proof}
    First, we fix an order on matching $E_{l_1,\ldots,l_{r-1}}\subset E$. We say $E_{l_1,\ldots,l_{r-1}}<E_{l_1',\ldots,l_{r-1}'}$ if $(l_1,\ldots,l_{r-1})<_{lex}(l_1',\ldots,l_{r-1}')$. Now, we will show the matching $E_{l_1,\ldots,l_{r-1}}$ is a positive matching on $\Lambda=(V,E\setminus \cup_{(l_1',\ldots,l_{r-1}')<_{lex}(l_1,\ldots,l_{r-1})} E_{l_1',\ldots,l_{r-1}'})$ by cardinality of the matching. 

    If $\lvert E_{l_1,\ldots,l_{r-1}}\rvert = 1$ then  from Remark \ref{Rem.pos}(b) it follows that $E_{l_1,\ldots,l_{r-1}}$ is positive matching on $\Lambda$. We depict a matching $E_{l_1,\ldots,l_{r-1}}$ with ($\lvert E_{l_1,\ldots,l_{r-1}}\rvert >1$) using the following representation:
    
    \textbf{Case I.} For an odd integer $r\geq 5$  we set
    \begin{center}
    \begin{equation*}
        \begin{bmatrix}
            l_1-l_2+l_3-\ldots -l_{r-1}+m-a & \cdots & l_{r-1}-m+a & m-a\\
            l_1-l_2+l_3-\ldots -l_{r-1}+m-a+1 &  \cdots & l_{r-1}-m+a-1 & m-a+1\\
            \vdots & \vdots &  \vdots & \vdots \\
            l_1-l_2+l_3-\ldots -l_{r-1}+m-1 &  \cdots &  l_{r-1}-m+1 & m-1\\
            l_1-l_2+l_3-\ldots -l_{r-1}+m&  \cdots &  l_{r-1}-m & m\\
        \end{bmatrix}
    \end{equation*} 
    \begin{equation*}
        =
        \begin{bmatrix}
            x_{11} &   \cdots & x_{1,r-1} & x_{1r}\\
            x_{21} &  \cdots & x_{2,r-1} & x_{2r}\\
            \vdots &  \vdots & \vdots & \vdots \\
            x_{a1} &  \cdots & x_{3,r-1} & x_{ar}\\
            x_{a+1,1}  & \cdots & x_{a+1,r-1} & x_{a+1,r}\\
        \end{bmatrix} 
    \end{equation*}
    \end{center}
    where, $2r\leq m \leq n$, $a \in \mathbb{N}$ such that  $x_{1,i-1} < 
    x_{1,i}$, $x_{a+1,j-1}<x_{a+1,j}$, and $x_{11} \geq 1$, where $i=3,5,\ldots, r$, and $j=2,4,\ldots,r-1$. In the above representation, each row is an edge of $E_{l_1,\ldots,l_{r-1}}$. Consider the following  ascending sequence:
    \begin{multline} \label{asc.sequ}
        x_{11},x_{21}, \ldots, x_{a+1,1},x_{a+1,2}, x_{a2}, \ldots, x_{12}, x_{13},\\ x_{23}, \dots,x_{a+1,i},
        x_{a+1,i+1}, \dots, x_{1j}, x_{1,j+1}, \ldots, x_{a+1,r},   
    \end{multline}
    where $i=3,5,\ldots,r-2$  and $j=4,6,\ldots,r-1$. We define a map $\phi: V(E_{l_1,\ldots,l_{r-1}}) \rightarrow \mathbb{Q}$ as
    \begin{equation} \label{r-mapping}
        \begin{split}
            \phi(x_{11})&=t,  \\
            \phi(x_{1i})=\phi(x_{1,i+1})&=-(1+\phi(x_{11})+\ldots+\phi(x_{1,i-1})+\phi(\alpha_1)+\ldots+\phi(\alpha_r-i)),\\
        \end{split}
    \end{equation}
where $\alpha_1$  to  $\alpha_{r-i}$ constitute a sequence of  $r-i$ consecutive elements that commence immediately after $x_{1,i+1}$  from Sequence (\ref{asc.sequ}), and $i=2,4,\ldots, r-1$, such that 
    \begin{equation} 
        \begin{split}
            \phi(x_{12})<0& \text{ and } \phi(x_{12})\geq\phi(x_{13})\geq\ldots\geq\phi(x_{1r}) \text{ and }  \phi(x_{11})+\ldots + \phi(x_{1r})=1, \\
            \phi(x_{\ell r})&=-(1+\phi(x_{\ell-1,1})+\ldots+\phi(x_{\ell-1,r-1})), \\
            \phi(x_{\ell k})&=-(1+\phi(x_{\ell 1})+\ldots+\phi(x_{\ell k-1})+\phi({\beta_1})+\ldots+\phi(\beta_{r-k})), \\
        \end{split}
    \end{equation}
where $\beta_1$  to  $\beta_{r-k}$ constitute a sequence of  $r-k$ consecutive elements that commence immediately after $x_{\ell k+1}$  from Sequence (\ref{asc.sequ}), $\phi(x_{\ell,0})=0$, and $1\leq k \leq r-1$,
    \begin{equation*} 
        \begin{split}
            \phi(x_{\ell-1,i})&>\phi(x_{\ell,i}), \text{ where } i=1,3,\ldots,r,  \\ 
            \phi(x_{\ell-1,j})&<\phi(x_{\ell,j}), \text{ where } j=2,4,\ldots,r-1, \\
            \text{ and } \phi(x_{\ell 1})&>\phi(x_{\ell 2})>\ldots>\phi(x_{\ell r}) \text{ such that }  \phi(x_{\ell 1})+\ldots + \phi(x_{\ell r})=1, \\
        \end{split}
    \end{equation*}
    where $2\leq \ell \leq a+1$, $t \gg 0$ such that  $\phi(x_{a+1,1}) > 0$ and $\phi(x_{a+1,2}) \leq 0$.

    The definition of $\phi$ implies that the following sequence is in descending order:
    \begin{multline} \label{dsc.sequ}
        \phi(x_{11}),\phi(x_{21}), \ldots, \phi(x_{a+1,1}),\phi(x_{a+1,2}), \phi(x_{a2}), \ldots, \phi(x_{12}), \phi(x_{13}),\\ \phi(x_{23}), \dots,\phi(x_{a+1,i}), \phi(x_{a+1,i+1}), \dots, \phi(x_{1j}), \phi(x_{1,j+1}), \ldots, \phi(x_{a+1,r}),  
    \end{multline}
    where $i=3,5,\ldots,r-2$  and $j=4,6,\ldots,r-1$.
    
 \textbf{Case II.} For an even integer $r\geq 4$  we set
    \begin{center}
    \begin{equation*}
        \begin{bmatrix}
            l_1-l_2+l_3-\ldots +l_{r-1}-m & \cdots & l_{r-1}-m & m\\
            l_1-l_2+l_3-\ldots +l_{r-1}-m+1 &  \cdots & l_{r-1}-m+1 & m-1\\
            \vdots & \vdots &  \vdots & \vdots \\
            l_1-l_2+l_3-\ldots +l_{r-1}-m+a-1 &  \cdots &  l_{r-1}-m+a-1 & m-a+1\\
            l_1-l_2+l_3-\ldots +l_{r-1}-m+a&  \cdots &  l_{r-1}-m+a & m-a\\
        \end{bmatrix}
    \end{equation*} 
    \begin{equation*}
        =
        \begin{bmatrix}
            x_{11} &   \cdots & x_{1,r-1} & x_{1r}\\
            x_{21} &  \cdots & x_{2,r-1} & x_{2r}\\
            \vdots &  \vdots & \vdots & \vdots \\
            x_{a1} &  \cdots & x_{3,r-1} & x_{ar}\\
            x_{a+1,1}  & \cdots & x_{a+1,r-1} & x_{a+1,r}\\
        \end{bmatrix} 
    \end{equation*}
    \end{center}
    where, $2r\leq m \leq n$, $a \in \mathbb{N}$ such that  $x_{1,i-1} < 
    x_{1,i}$, $x_{a+1,j-1}<x_{a+1,j}$, and $x_{11} \geq 1$, where $i=3,5,\ldots,r-1$, and $j=2,4,\ldots,r$. In the above representation, each row is an edge of $E_{l_1,\ldots,l_{r-2}}$. Consider the following  ascending sequence:
    \begin{multline} \label{2asc.sequ}
        x_{11},x_{21}, \ldots, x_{a+1,1},x_{a+1,2}, x_{a2}, \ldots, x_{12}, x_{13},\\ x_{23}, \dots,x_{a+1,i},
        x_{a+1,i+1}, \dots, x_{1j}, x_{1,j+1}, \ldots, x_{a+1,r-1}, x_{a+1,r},\ldots, x_{1,r},   
    \end{multline}
    where $i=3,5,\ldots,r-3$  and $j=4,6,\ldots,r-2$. We define a map $\psi: V(E_{l_1,\ldots,l_{r-1}}) \rightarrow \mathbb{Q}$ as
     \begin{equation} \label{r2-mapping}
        \begin{split}
            \psi(x_{11})&=t,  \\
            \psi(x_{1i})=\psi(x_{1,i+1})&=-(1+\psi(x_{11})+\ldots+\psi(x_{1,i-1})+\psi(\alpha_1)+\ldots+\psi(\alpha_r-i)),\\
        \end{split}
    \end{equation}
where $\alpha_1$  to  $\alpha_{r-i}$ constitute a sequence of  $r-i$ consecutive elements that commence immediately after $x_{1,i+1}$  from Sequence (\ref{2asc.sequ}), and $i=2,4,\ldots, r-2$, such that 
    \begin{equation} 
        \begin{split}
            \psi(x_{12})<0& \text{ and } \psi(x_{12})\geq\psi(x_{13})\geq\ldots\geq\psi(x_{1r}) \text{ and }  \psi(x_{11})+\ldots + \psi(x_{1r})=1, \\
            \psi(x_{m r})&=-(1+\psi(x_{m+1,1})+\ldots+\psi(x_{m+1,r-1})), \text{ where } m=1,\ldots,a,\\
            \psi(x_{\ell k})&=-(1+\psi(x_{\ell 1})+\ldots+\psi(x_{\ell k-1})+\psi({\beta_1})+\ldots+\psi(\beta_{r-k})), \\
        \end{split}
    \end{equation}
where $\beta_1$  to  $\beta_{r-k}$ constitute a sequence of  $r-k$ consecutive elements that commence immediately after $x_{\ell k+1}$  from Sequence (\ref{2asc.sequ}), $\psi(x_{\ell,0})=0$, and $1\leq k \leq r-1$,
    \begin{equation*} 
        \begin{split}
            \psi(x_{\ell-1,i})&>\psi(x_{\ell,i}), \text{ where } i=1,3,\ldots,r-1,  \\
            \psi(x_{\ell-1,j})&<\psi(x_{\ell,j}), \text{ where } j=2,4,\ldots,r, \\
            \text{ and } \psi(x_{\ell 1})&>\psi(x_{\ell 2})>\ldots>\psi(x_{\ell r}) \text{ such that }  \psi(x_{\ell 1})+\ldots + \psi(x_{\ell r})=1, \\
        \end{split}
    \end{equation*}
    where $2\leq \ell \leq a+1$, $t \gg 0$ such that  $\psi(x_{a+1,1}) > 0$ and $\psi(x_{a+1,2}) \leq 0$.  

    The definition of $\phi$ implies that the subsequent sequence is in descending order.
    \begin{multline} \label{2dsc.sequ}
        \psi(x_{11}),\psi(x_{21}), \ldots, \psi(x_{a+1,1}),\psi(x_{a+1,2}), \psi(x_{a2}), \ldots, \psi(x_{12}), \psi(x_{13}), \psi(x_{23}), \dots,\psi(x_{a+1,i}), \\ \psi(x_{a+1,i+1}), \dots, \psi(x_{1j}), \psi(x_{1,j+1}), \ldots, \psi(x_{a+1,r-1}), \psi(x_{a+1,r}),\ldots, \psi(x_{1,r}),   
    \end{multline}
    where $i=3,5,\ldots,r-3$  and $j=4,6,\ldots,r-2$.

    Let $E_{l_1,\ldots,l_{r-1}}^c$ be the set all edges in $(V_{E_{l_1,\ldots,l_{r-1}}},E\setminus \{\cup_{(l_1',\ldots,l_{r-1}')<_{lex}(l_1,\ldots,l_{r-1})}E_{l_1',\ldots,l_{r-1}'} \cup E_{l_1,\ldots,l_{r-1}}\})$. From maps $\phi$ and $\psi$ it is clear that $\phi(x_{i1})+\ldots+\phi(x_{ir})=1$ and $\psi(x_{i1})+\ldots+\psi(x_{ir})=1$ for all $i=1,\ldots,a+1$. From maps $\phi$ and $\psi$, and referring to Equations (\ref{dsc.sequ}) and (\ref{2dsc.sequ}) one has $\sum_{i\in e}\phi(i)<0$ and $\sum_{i\in e}\psi(i)<0$ for all $e\in E_{l_1,\ldots,l_{r-1}}^c$, analogous to the proof presented in Claim \ref{claim2} of Theorem \ref{conjpmd}. Hence the matching $E_{l_1,\ldots,l_{r-1}}$ is a positive matching on $\Lambda$.
\end{proof} 

\begin{remark}
    In Theorem \ref{thm.runipmd}, it is necessary to establish the definitions of the maps $\phi$ and $\psi$ to facilitate matching with cardinalities that are less than or equal to $r-2$. In cases where $\lvert E_{l_1,\ldots,l_{r-1}} \rvert > r-2$, these maps can be recursively derived.
\end{remark}

\begin{theorem} \label{thm.runibound}
    Let $H=(V,E)$ be an $r$-uniform hypergraph with $n$ vertices. Then $\pmd(H) \leq n^{r-1}+O(r-2)$.
\end{theorem}
\begin{proof}
    It follows from Proposition \ref{prop.runibound} and Theorem \ref{thm.runipmd}.
\end{proof}
\begin{cor} 
    Let $H$ be an $r$-uniform hypergraph. If $d \geq n^{r-1}+O(r-2)$, then:
    \begin{enumerate}
        \item $L_H^{\mathbb{K}}(d)$ is a radical complete intersection;
        \item $L_H^{\mathbb{K}}(d+1)$ is a prime.
    \end{enumerate}
\end{cor}

\section{Positive matching decomposition} \label{secpmd}
Generally characterising $\pmd$s of a hypergraph is a challenging problem. To begin with, our focus is on comprehending the $\pmd$s within the context of $r$-uniform hypergraph $H=(V,E)$ such that $\mid e_i \cap e_j \mid \leq 1$, where $e_i,e_j \in E$, for all $i,j$.

A \textit{walk} in a $H$ is a sequence of alternate vertices and edges $u_1e_1u_2\ldots e_nu_n$ such that $u_i,u_{i+1} \in e_i$ for all $i$.  
A \textit{matching} in a graph $H$ is a set $M=\{e_1,\ldots, e_n\} \subset E$ such that $e_i \cap e_j = \emptyset$ for all $i,j$.  For $A \subseteq V, H[A]$ denotes the \textit{induced subgraph} of $H$ on the vertex set $A$, that is, $H[A]=(A,\{e\in E \mid e \subseteq A\})$. Throughout this section we assume $H = (V,E)$ be an $r$-uniform hypergraph such that $\mid e_i \cap e_j \mid \leq 1$, where $e_i,e_j \in E$, for all $i,j$, unless otherwise we stated.

\begin{definition} 
    An \textit{alternating walk} in a hypergraph $H$ with respect to a matching $M$ is a walk whose edges alternate between edges of $M$ and $E\setminus M$. 
\end{definition}

\begin{definition}
    Let $H$ be a $3$-uniform hypergraph and $M$ be a matching of $H$. An alternate closed walk in $H$ with respect to $M$ say $W=u_1,a_1,v_1,b_1,u_2,a_2,v_2,b_2, \ldots, v_n,b_n,u_{n+1} = u_1,$ where $u_i,v_i,a_i,b_i \in V$, $\{u_i,a_i,v_i\} \in M$ and $\{v_i,b_i,u_{i+1}\} \in E\setminus M$, is said to be a \textit{strong alternate closed walk} if $a_i=b_j$ for some $j$ and number of times $a_i$ appears in $W$ is equal to number of times $b_j$ appears in $W$ for all $i$.
\end{definition}

\begin{remark}
    Certainly, the definition of a strong alternate closed walk can be naturally extended to any $r$-uniform hypergraph. Moreover, it becomes evident that strong alternate closed walks and alternate closed walks coincide when $r=2$.
\end{remark}

The following theorem gives a characterization of positive matching via strong alternate closed walks. It turns out that a matching $M$ of a hypergraph $H$ is positive if and only if the subgraph induced by $M$ has no strong alternate closed walks with respect to $M$.

\begin{theorem} \label{thm.3-uni}
    Let $H=(V,E)$ be a $3$-uniform hypergraph and $M \subset E$ be a matching. The following conditions are equivalent:
    \begin{enumerate}
        \item $M$ is positive;
        \item The subgraph induced by $M$ does not contain any strong alternate closed walk;
        \item The subgraph induced by every subset $V(N) \subseteq V(M)$ does not contain any subgraph $N_1 \subseteq H[V(N)]$ such that $N \subseteq N_1$ and $\deg_{N_1}(u_i) =\deg_{N_1}(a_i)=\deg_{N_1}(v_i)= k$, where $k\geq 2$ for all $\{u_i,a_i,v_i\} \in N$.    
    \end{enumerate}
\end{theorem}

\begin{proof}
    $(1) \implies (2)$. Suppose $H[V(M)]$ has a strong alternate closed walk with the following vertices $$u_1,a_1,v_1,b_1,u_2,a_2,v_2,b_2, \ldots, v_n,b_n,u_{n+1} = u_1,$$ where $\{u_i,a_i,v_i\} \in M$ and $\{v_i,b_i,u_{i+1}\} \in E(H[V(M)])\setminus M$, for all $i$. Since $M$ is positive there exists a map $\rho: V(M) \rightarrow \mathbb{Z}$ such that $\rho(u_i)+\rho(a_i)+\rho(v_i) > 0$ and $\rho(v_i)+\rho(b_i)+\rho(u_{i+1}) < 0$, for all $i$. WLOG, we can assume that $\rho(u_1)>0$ as $\rho(u_1)+\rho(a_1)+\rho(v_1)>0$. From the positive matching of $M$, we get the following inequalities:
    \begin{equation*}
    \begin{split}
    \rho(u_1) &< -(\rho(v_n)+\rho(b_n)), \\ 
    \rho(u_1) &< -(\rho(v_n)+\rho(b_n)) < \rho(u_n)+\rho(a_n)-\rho(b_n), \\   
    \rho(u_1) &< -(\rho(v_n)+\rho(b_n)) < \rho(u_n)+\rho(a_n)-\rho(b_n) < -(\rho(v_{n-1})+\rho(b_{n-1})-\rho(a_n)+\rho(b_n))), \\
    \vdots & \\
    \rho(u_1) &< \rho(u_1)+\sum_{i=n}^{n}\rho(a_i)-\sum_{i=n}^{n}\rho(b_i), \\
    \end{split}
    \end{equation*}
    which is a contradiction as $a_i=b_j$, for some $j$ and number of times $a_i$ appears is equal to number of times $b_j$ appears for all $i$.

    $(2) \implies (3)$. It follows from Lemma \ref{Lemma.2-3}.

    $(3) \implies (1)$. Assume that $M$ is not positive. Let $N \subset M$ be a minimal such that $N$ is not positive. Let $N=\{e_1,\ldots,e_n\}$ and $N^c = E(H[V(N)])\setminus N = \{e'_1,\ldots, e'_m\}$ be edges and $e_i=\{u_{i1},u_{i2},u_{i3}\}$ for all $i$. Let $\rho : V(N) \rightarrow \mathbb{Z}$ be a map. By not positiveness of $N$ one has $\sum_{u_i \in e_i} \rho(u_i) = y_i$ and $\sum_{u_j \in e'_j} \rho(u_j) = -y_j$, where $y_i,y_j >0$, for all $1\leq i \leq n$ and $1\leq j \leq m$, has no solution. Now we represent the system of linear equations using a matrix by $Ax=Y$, where $x=[\rho(u_{11}),\ldots,\rho(u_{n3})]^t$,  $Y=[y_1,\ldots, y_{n+m}]$. Notice that each row of $A$ corresponds to an edge, and the number of non-zero entries in each column of $A$ corresponds to the degree of the respective vertex. Then we have $\rank(A,Y) > \rank(A)$. Let $R_i$ denotes $i^{th}$ row of matrix $A$. First, $n$ rows of matrix $A$ corresponds to edges of $N$ and rows from $n+1$ to $n+m$ correspond to edges of $N^c$. Therefore we get $a_1R_1 + \ldots + a_nR_n = a_{n+1}R_{n+1} + \ldots + a_{n+m}R_{n+m}$, where $a_i>0$ be a positive integer for all $1 \leq i \leq n$ and $a_j$ be a non negative integer for all $n+1 \leq j \leq n+m$. Then we get $\deg_{N_1}(u_{i1})=\deg_{N_1}(u_{i2})=\deg_{N_1}(u_{i3})=a_i +1$, for all $i$, where edge set of $N_1=\{N \cup \{e'_j \mid a_j\neq 0 \text{ for } n+1 \leq j \leq n+m\}\}$, and $e'_j$  is the edge associated with $R_j$. Hence $H[V(N)]$ has a subgraph $N_1$ such that $N \subseteq N_1$ and $\deg_{N_1}(u_{i1}) =\deg_{N_1}(u_{i2})=\deg_{N_1}(u_{i3})= k$, where $k\geq 2$ for all $\{u_{i1},u_{i2},u_{i3}\} \in N$.
\end{proof}

\begin{notation}
    We write the vertices of an edge as an ordered sequence, and we call the first vertex in the sequence the parent and the last vertex we call the descendant.
\end{notation}

\begin{example}
    Let $H$ be a $3$-uniform hypergraph. For the edge $e=\{1,2,3\} \in E(H)$, consider writing the vertices of $e$ as an ordered sequence $(2,3,1)$. In this sequence, $2$ is the parent and $1$ is the descendant. If we fix a vertex $u$ of the edge $e$, there are exactly two sequences where $u$ is the parent vertex.
\end{example}

In the following, we present a construction on a hypergraph to generate a strong alternate closed walk.

\subsection{Construction} \label{const}
Let $H$ be a $3$-uniform hypergraph and $M$ be a matching of $H$. Let $N\subseteq M$  and $N_1$ be a subgraph of $H[V(N)]$ such that $N \subset N_1$ and the degree of each vertex in an edge is equal and greater than or equal to $2$, for all edges in $N$, i.e. $\deg(u_i) =\deg(a_i)=\deg(v_i)= k$, where $k\geq 2$ for all $\{u_i,a_i,v_i\}=e_i \in N$. Assume that $N$ and $N_1$ are minimal with this property. Let $N = \{e_1,\ldots, e_n\}$ and $N^c=\{e \mid e \in N_1\setminus N\} = \{e'_1,\ldots,e'_m\}$ be edges.

The following construction resembles a rooted tree structure with alternating edges with respect to $N$, and edges of $N$ may repeat. An alternate rooted tree $(H[V(N)],N,u_0)$ is an alternate walk with respect to $N$ with root $u_0$.

\begin{enumerate}
    \item Let $u_0\in V(N)$ be a vertex. Then there exists a unique edge in $N$, which intersects with $u_0$. Let $e_1=\{u_0,a_1,v_1\}$ be an edge in $N$. By fixing the vertex $u_0$, there are exactly two possible sequences where $u_0$ is the parent. This leaves two possibilities for the descendant vertex, as $u_0$ serves as the parent vertex in both sequences. The two descendants of $u_0$ are namely $v_1$ and $a_1$ corresponding to the ordered sequences $u_0,a_1,v_1$ and $u_0,v_1,a_1$, respectively.
    
    \item We will repeat the following process for each $2n^{th}$ step. For a descendant vertex, say $v$, if $2(\deg_{H[V(N)]}(v)-1)$ is greater than the number of times the vertex $v$ appeared in the alternate walk from the root vertex $u_0$ to the descendant vertex $v$ with respect to $N$ in the construction, then there exists at least one edge in $N^c$ which intersects with $v$ and the edge does not appear in the alternate walk from the root vertex $u_0$ to $v$, Now fix $v$, 
    there are exactly two possible sequences where $v$ is the parent. This leaves two possibilities for descendant vertex.  Similarly, we have two descendants for each edge in $N^c$, which intersects with $v$. Otherwise, $v$ is the last descendant.
    \item We will repeat the following process for each $2n+1^{th}$ step. For each descendant vertices, say $u$, if $2(\deg_{H[V(N)]}(u)-1)$ is greater than the number of times $u$ appeared in the alternate walk from the root vertex $u_0$  to the descendant vertex $u$, then there exists a unique edge in $N$ which intersects with $u$. Then, we fix $u$, there are exactly two possible sequences where $u$ is the parent. This leaves two possibilities for descendant vertex. Otherwise, $u$ is the last descendant.
\end{enumerate}

\begin{example}
    Let $H=(V,E)$ be a hypergraph on $[9]$. Let $M=\{\{1,2,3\},\{4,5,6\},\{7,8,9\}\}$ be a matching and $ M^c=\{\{1,4,7\},\{2,5,8\},\{3,6,9\}\}$.

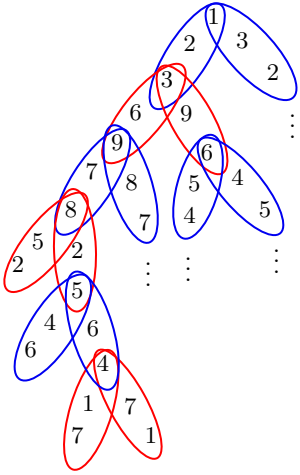
\begin{multicols}{2}

\tikzset{every picture/.style={line width=0.75pt}} %set default line width to 0.75pt        

\begin{tikzpicture}[x=0.6pt,y=0.6pt,yscale=-1,xscale=1]

%Shape: Ellipse [id:dp5877786816544748] 
\draw   [color={rgb, 255:red, 0; green, 0; blue, 240 }  ,draw opacity=1 ] (287.9,558.2) .. controls (281.52,554.8) and (285.48,537.41) .. (296.75,519.36) .. controls (308.01,501.32) and (322.32,489.45) .. (328.7,492.86) .. controls (335.08,496.26) and (331.11,513.65) .. (319.85,531.69) .. controls (308.58,549.74) and (294.28,561.61) .. (287.9,558.2) -- cycle ;
%Shape: Ellipse [id:dp7399508850878751] 
\draw [color={rgb, 255:red, 0; green, 0; blue, 240 }  ,draw opacity=1 ]  (375.72,549.55) .. controls (372.22,557.41) and (357.02,551.77) .. (341.77,536.93) .. controls (326.52,522.1) and (316.99,503.7) .. (320.49,495.84) .. controls (323.99,487.97) and (339.19,493.62) .. (354.44,508.45) .. controls (369.69,523.28) and (379.22,541.68) .. (375.72,549.55) -- cycle ;
%Shape: Ellipse [id:dp3144250347699439] 
\draw  [color={rgb, 255:red, 254; green, 0; blue, 0 }  ,draw opacity=1 ] (257.62,591.91) .. controls (251.63,587.87) and (257.37,570.98) .. (270.45,554.2) .. controls (283.52,537.42) and (298.98,527.1) .. (304.97,531.14) .. controls (310.96,535.19) and (305.22,552.07) .. (292.15,568.86) .. controls (279.07,585.64) and (263.62,595.96) .. (257.62,591.91) -- cycle ;
%Shape: Ellipse [id:dp8596851075354907] 
\draw  [color={rgb, 255:red, 255; green, 0; blue, 0 }  ,draw opacity=1 ] (330.49,598.92) .. controls (324.48,602.94) and (311.01,591.25) .. (300.41,572.81) .. controls (289.81,554.36) and (286.09,536.15) .. (292.1,532.13) .. controls (298.11,528.11) and (311.58,539.8) .. (322.18,558.25) .. controls (332.78,576.69) and (336.5,594.9) .. (330.49,598.92) -- cycle ;
%Shape: Ellipse [id:dp7908695457086041] 
\draw [color={rgb, 255:red, 0; green, 0; blue, 240 }  ,draw opacity=1 ]  (228.9,637.2) .. controls (222.52,633.8) and (226.48,616.41) .. (237.75,598.36) .. controls (249.01,580.32) and (263.32,568.45) .. (269.7,571.86) .. controls (276.08,575.26) and (272.11,592.65) .. (260.85,610.69) .. controls (249.58,628.74) and (235.28,640.61) .. (228.9,637.2) -- cycle ;
%Shape: Ellipse [id:dp23013793590955933] 
\draw  [color={rgb, 255:red, 0; green, 0; blue, 240 }  ,draw opacity=1 ] (284.77,642.55) .. controls (278.05,645.21) and (267.34,630.95) .. (260.85,610.69) .. controls (254.36,590.43) and (254.55,571.85) .. (261.27,569.18) .. controls (267.99,566.51) and (278.7,580.77) .. (285.19,601.03) .. controls (291.68,621.29) and (291.49,639.88) .. (284.77,642.55) -- cycle ;
%Shape: Ellipse [id:dp3803316840383061] 
\draw  [color={rgb, 255:red, 0; green, 0; blue, 240 }  ,draw opacity=1 ] (306.36,640.17) .. controls (299.28,638.79) and (297.43,623.08) .. (302.22,605.09) .. controls (307.02,587.09) and (316.64,573.62) .. (323.72,575) .. controls (330.8,576.38) and (332.65,592.09) .. (327.85,610.08) .. controls (323.06,628.08) and (313.43,641.55) .. (306.36,640.17) -- cycle ;
%Shape: Ellipse [id:dp004300121503566401] 
\draw  [color={rgb, 255:red, 0; green, 0; blue, 240 }  ,draw opacity=1 ] (366.53,636.6) .. controls (361.36,641.66) and (345.98,632.64) .. (332.16,616.46) .. controls (318.35,600.28) and (311.34,583.07) .. (316.51,578.01) .. controls (321.68,572.95) and (337.07,581.97) .. (350.88,598.15) .. controls (364.69,614.33) and (371.7,631.54) .. (366.53,636.6) -- cycle ;
%Shape: Ellipse [id:dp47288157257949603] 
\draw  [color={rgb, 255:red, 254; green, 0; blue, 0 }  ,draw opacity=1 ] (196.62,672.91) .. controls (190.63,668.87) and (196.37,651.98) .. (209.45,635.2) .. controls (222.52,618.42) and (237.98,608.1) .. (243.97,612.14) .. controls (249.96,616.19) and (244.22,633.07) .. (231.15,649.86) .. controls (218.07,666.64) and (202.62,676.96) .. (196.62,672.91) -- cycle ;
%Shape: Ellipse [id:dp3746973391808416] 
\draw  [color={rgb, 255:red, 254; green, 0; blue, 0 }  ,draw opacity=1 ] (239.9,686.01) .. controls (232.71,686.8) and (226.17,670.21) .. (225.28,648.96) .. controls (224.4,627.7) and (229.51,609.83) .. (236.7,609.04) .. controls (243.88,608.25) and (250.43,624.84) .. (251.31,646.1) .. controls (252.19,667.35) and (247.08,685.22) .. (239.9,686.01) -- cycle ;
%Shape: Ellipse [id:dp12328975622648985] 
\draw  [color={rgb, 255:red, 254; green, 0; blue, 0 }  ,draw opacity=1 ] (237.35,785.46) .. controls (230.26,783.88) and (229.49,765.77) .. (235.61,745.01) .. controls (241.74,724.25) and (252.45,708.7) .. (259.53,710.28) .. controls (266.62,711.85) and (267.39,729.96) .. (261.27,750.72) .. controls (255.14,771.49) and (244.43,787.04) .. (237.35,785.46) -- cycle ;
%Shape: Ellipse [id:dp19223370978873722] 
\draw  [color={rgb, 255:red, 254; green, 0; blue, 0 }  ,draw opacity=1 ] (289.82,778.06) .. controls (284.31,781.6) and (271.53,769.36) .. (261.27,750.72) .. controls (251.01,732.09) and (247.15,714.12) .. (252.66,710.58) .. controls (258.16,707.04) and (270.94,719.28) .. (281.21,737.91) .. controls (291.47,756.54) and (295.32,774.52) .. (289.82,778.06) -- cycle ;
%Shape: Ellipse [id:dp30019457070823374] 
\draw  [color={rgb, 255:red, 0; green, 0; blue, 240 }  ,draw opacity=1 ] (204.1,728.78) .. controls (197.28,725.14) and (200.88,707.56) .. (212.15,689.51) .. controls (223.42,671.47) and (238.08,659.79) .. (244.9,663.43) .. controls (251.72,667.07) and (248.11,684.65) .. (236.85,702.69) .. controls (225.58,720.74) and (210.92,732.42) .. (204.1,728.78) -- cycle ;
%Shape: Ellipse [id:dp8373529263134252] 
\draw  [color={rgb, 255:red, 0; green, 0; blue, 240 }  ,draw opacity=1 ] (260.2,734.96) .. controls (253.43,737.51) and (242.98,723.06) .. (236.85,702.69) .. controls (230.72,682.32) and (231.23,663.74) .. (238,661.19) .. controls (244.77,658.64) and (255.22,673.09) .. (261.35,693.46) .. controls (267.48,713.83) and (266.97,732.41) .. (260.2,734.96) -- cycle ;

% Text Node
\draw (319,494) node [anchor=north west][inner sep=0.75pt]   [align=left] {\footnotesize 1};
% Text Node
\draw (304,510.8) node [anchor=north west][inner sep=0.75pt]   [align=left] {\footnotesize 2};
% Text Node
\draw (356,528.8) node [anchor=north west][inner sep=0.75pt]   [align=left] {\footnotesize 2};
% Text Node
\draw (337,508.8) node [anchor=north west][inner sep=0.75pt]   [align=left] {\footnotesize 3};
% Text Node
\draw (290,533) node [anchor=north west][inner sep=0.75pt]   [align=left] {\footnotesize 3};
% Text Node
\draw (270.45,554.2) node [anchor=north west][inner sep=0.75pt]   [align=left] {\footnotesize 6};
% Text Node
\draw (302,555) node [anchor=north west][inner sep=0.75pt]   [align=left] {\footnotesize 9};
% Text Node
\draw (315,579) node [anchor=north west][inner sep=0.75pt]   [align=left] {\footnotesize 6};
% Text Node
\draw (259,573) node [anchor=north west][inner sep=0.75pt]   [align=left] {\footnotesize 9};
% Text Node
\draw (243,591) node [anchor=north west][inner sep=0.75pt]   [align=left] {\footnotesize 7};
% Text Node
\draw (230,614.8) node [anchor=north west][inner sep=0.75pt]   [align=left] {\footnotesize 8};
% Text Node
\draw (268,598) node [anchor=north west][inner sep=0.75pt]   [align=left] {\footnotesize 8};
% Text Node
\draw (276.04,623.18) node [anchor=north west][inner sep=0.75pt]   [align=left] {\footnotesize 7};
% Text Node
\draw (307,599) node [anchor=north west][inner sep=0.75pt]   [align=left] {\footnotesize 5};
% Text Node
\draw (304,619) node [anchor=north west][inner sep=0.75pt]   [align=left] {\footnotesize 4};
% Text Node
\draw (334,595) node [anchor=north west][inner sep=0.75pt]   [align=left] {\footnotesize 4};
% Text Node
\draw (351,615) node [anchor=north west][inner sep=0.75pt]   [align=left] {\footnotesize 5};
% Text Node
\draw (209.45,635.2) node [anchor=north west][inner sep=0.75pt]   [align=left] {\footnotesize 5};
% Text Node
\draw (197,649.8) node [anchor=north west][inner sep=0.75pt]   [align=left] {\footnotesize 2};
% Text Node
\draw (234,640.8) node [anchor=north west][inner sep=0.75pt]   [align=left] {\footnotesize 2};
% Text Node
\draw (234,665.8) node [anchor=north west][inner sep=0.75pt]   [align=left] {\footnotesize 5};
% Text Node
\draw (217,685.8) node [anchor=north west][inner sep=0.75pt]   [align=left] {\footnotesize 4};
% Text Node
\draw (205,702.8) node [anchor=north west][inner sep=0.75pt]   [align=left] {\footnotesize 6};
% Text Node
\draw (243.9,690.01) node [anchor=north west][inner sep=0.75pt]   [align=left] {\footnotesize 6};
% Text Node
\draw (250,711.8) node [anchor=north west][inner sep=0.75pt]   [align=left] {\footnotesize 4};
% Text Node
\draw (241,736.8) node [anchor=north west][inner sep=0.75pt]   [align=left] {\footnotesize 1};
% Text Node
\draw (234,757) node [anchor=north west][inner sep=0.75pt]   [align=left] {\footnotesize 7};
% Text Node
\draw (267.2,739) node [anchor=north west][inner sep=0.75pt]   [align=left] {\footnotesize 7};
% Text Node
\draw (280,757) node [anchor=north west][inner sep=0.75pt]   [align=left] {\footnotesize 1};

\draw (280,643) node [anchor=north west][inner sep=0.75pt]   [align=left] {\footnotesize $\vdots$};

\draw (305,640) node [anchor=north west][inner sep=0.75pt]   [align=left] {\footnotesize $\vdots$};

\draw (360,635) node [anchor=north west][inner sep=0.75pt]   [align=left] {\footnotesize $\vdots$};

\draw (370,550) node [anchor=north west][inner sep=0.75pt]   [align=left] {\footnotesize $\vdots$};

\end{tikzpicture}

Observe that $\deg_H(i)=2$ for all $i \in V$. We will now generate a strong alternate closed walk using an alternate rooted tree. Let $(H,M,1)$ be an alternate rooted tree with root $1$. Then the alternating edges $\{1,2,3\}$, $\{3,6,9\}$, $\{9,7,8\}$, $\{8,2,5\}$, $\{5,6,4\}$ and $\{4,7,1\}$ forms a strong alternate closed walk in $H$. And the alternate edges $\{1,2,3\}$, $\{3,6,9\}$, $\{9,8,7\}$, and $\{7,4,1\}$ is an alternate closed walk but not a strong alternate closed walk. Consequently, based on the construction, it can be deduced that the longest alternate walk with the same starting and ending vertex yields a strong alternate closed walk in $H$.
\end{multicols}

\end{example}

\begin{lemma} \label{Lemma.2-3}
    Let $H$ be a hypergraph and $M \subset E$ be a matching. If the subgraph induced by $M$ does not contain any strong alternate closed walk, then the subgraph induced by every subset $V(N) \subseteq V(M)$ does not contain any subgraph $N_1 \subseteq H[V(N)]$ such that $N \subseteq N_1$ and $\deg_{N_1}(u_i) =\deg_{N_1}(a_i)=\deg_{N_1}(v_i)= k$, where $k\geq 2$ for all $\{u_i,a_i,v_i\} \in N$.
\end{lemma}
\begin{proof}
    Suppose there exists a subset $V(N) \subseteq V(M)$ such that $H[V(N)]$ has a subgraph $N_1 \supseteq N$ with $\deg_{N_1}(u_i) =\deg_{N_1}(a_i)=\deg_{N_1}(v_i)= k$, where $k\geq 2$ for all $\{u_i,a_i,v_i\} \in N$. Assume that $N$ and $N_1$ are minimal with this property. Now, we construct an alternating closed walk in a subgraph $N_1$ with respect to $N$. From the construction \ref{const}, it follows that there exists a strong alternate closed walk in $(H[V(N)],N,u_1)$ such that $\deg_{N_1}(u_i) =\deg_{N_1}(a_i)=\deg_{N_1}(v_i)= k$, where $k\geq 2$ for all $\{u_i,a_i,v_i\} \in N$, which is a contradiction. 
\end{proof}

The concept of extension for $r$-uniform hypergraphs can be applied analogously, enabling us to broaden the understanding of their structural properties and relationships. 

\begin{theorem} \label{thm.r-uni}
    Let $H=(V,E)$ be an $r$-uniform hypergraph and $M \subset E$ be a matching. The following conditions are equivalent:
    \begin{enumerate}
        \item $M$ is positive;
        \item The subgraph induced by $M$ does not contain any strong alternate closed walk;
        \item The subgraph induced by every subset $V(N) \subseteq V(M)$ does not contain any subgraph $N_1 \subseteq H[V(N)]$ such that $N \subseteq N_1$ and $\deg_{N_1}(u_{i1}) =\deg_{N_1}(u_{i2})=\ldots=\deg_{N_1}(u_{ir})= k$, where $k\geq 2$ for all $\{u_{i1},\ldots,u_{ir}\} \in N$.  
    \end{enumerate}
\end{theorem}
\begin{proof}
    The proof is similar to the proof of Theorem \ref{thm.3-uni}.
\end{proof}

\begin{remark}
    Note that we obtain \cite[Theorem 2.1]{FG2022} as one specific case of Theorem \ref{thm.r-uni}.
\end{remark}

In the following theorem, we provide necessary and sufficient conditions for a matching with carnality equal to $2$ to be a positive matching in an arbitrary $r$-uniform hypergraph.

\begin{theorem} \label{pos.M_2}
    Let $H=(V,E)$ be an arbitrary $r$-uniform hypergraph. Let $M \subset E$ be a matching with cardinality of $M$ equal to $2$. Set $M=\{\{x_1,\ldots,x_r\},\{x_{r+1},\ldots,x_{2r}\}\}$, where $x_i$'s are distinct vertices of $H$. Then $M$ is positive if and only if an edge $\{x_{i_1},\ldots, x_{i_r}\} \in E(H[V(M)])\setminus M$ then $V(M) \setminus \{x_{i_1},\ldots, x_{i_r}\} \notin E(H[V(M)])$.
\end{theorem}
\begin{proof}
    Suppose $M$ is positive and edges $\{x_{i_1},\ldots, x_{i_r}\}, V(M) \setminus \{x_{i_1},\ldots, x_{i_r}\} \in E(H[V(M)])$. By positiveness of $M$, there exists a weight function $\mu$ such that it satisfies Equation \ref{pmd_condition}. This implies that $\sum_{i=1}^{2r}\mu(x_i) > 0$ and $\sum_{i=1}^{2r}\mu(x_i) < 0$, since $\{x_{i_1},\ldots, x_{i_r}\}$ and  $V(M) \setminus \{x_{i_1},\ldots, x_{i_r}\}$ are edges of $H[V(M)]$. This contradicts the positivity of $M$.

    Conversely, suppose an edge $\{x_{i_1},\ldots, x_{i_r}\} \in E(H[V(M)])$ and $V(M) \setminus \{x_{i_1},\ldots, x_{i_r}\} \notin E(H[V(M)])$. Then we show that $M$ is positive by induction on the number of edges in $E(H[V(M)]) \setminus M$. Let $t$ be the number of edges in $E(H[V(M)]) \setminus M$. For the case $t=1$, the statement follows from Remark \ref{Rem.pos}(b). Now, assume the statement is true for $t=s-1$. Suppose for $t=s$, $M$ is not positive. By non positiveness of $M$ one has, for every map $\rho$, the $\sum_{u_i \in e_i} \rho(u_i) = y_i$ and $\sum_{u_j \in e'_j} \rho(u_j) = -y_j$, where $y_i,y_j >0$, $e_i \in M$, $e_j' \in E(H[V(M)])\setminus M$ for all $1\leq i \leq 2$ and $3\leq j \leq s+2$, has no solution. Now we represent the system of linear equations using a matrix by $Ax=Y$, where $x=[\rho(x_1),\ldots,\rho(x_{2r})]^t$,  $Y=[y_1,\ldots, y_{s+2}]$. Then we have $\rank(A,Y) > \rank(A)$. Let $R_i$ denotes $i^{th}$ row of matrix $A$. First, $2$ rows of matrix $A$ correspond to edges of $M$, and rows from $3$ to $s+2$ correspond to edges of $E(H[V(M)]) \setminus M$. By induction hypothesis, it follows that removing a row, say $R_k$, where $3 \leq k \leq s+2$, then there exists a $\rho$ such that $\rank(A\setminus R_k,Y) = \rank(A \setminus R_k)$. Therefore, we get $a_1R_1 + a_2R_2 = (a_{3}R_{3} + \ldots \hat{a_{k}R_{k}}+ \ldots + a_{s+2}R_{s+2}) + a_{k}R_{k}$, where $a_1,a_2,a_k>0$ be a positive integers and $a_j$ be a non negative integer for all $3 \leq j \neq k \leq s+2$. Let $R_{k}$ correspond to an $\{x_{i_1},\ldots, x_{i_r}\}$. Then the rows $R_3, \ldots, \hat{R_K}, \ldots, R_{s+2}$ has an edge correspond to $V(M) \setminus \{x_{i_1},\ldots, x_{i_r}\} \in E(H[V(M)])$ which is a contradiction. Thus, we get $M$ is positive as desired.
\end{proof}

\begin{definition}
    Let $H=(V,E)$ be an $r$-uniform hypergraph on $[n]$. Then $H$ is called 
    \begin{enumerate}
        \item \textit{good forest} if there exist a sequence of edges $e_1,\dots,e_m$ such that $\lvert V_{\{e_1,\ldots,e_i\}}\cap V_{e_{i+1}}\rvert \leq 1$, for all $1\leq i \leq m-1$, where $m=\lvert E\rvert$. We call it a good tree if $H$ is connected.
        \item \textit{loose cycle} if $E=\{\{1,\dots,r\},\{r,\dots,2r-1\},\ldots,\{n-r-2,\dots,n,1\}\}$, denote it by $C_{(r-1)m}$, where $m>1$ and $(r-1)m=n$.
    \end{enumerate}
\end{definition}

\begin{remark}
    If $H=(V,E)$ be a good forest then $e_i\cap e_j\leq 1$, for all $e_i,e_j\in E$. Note that good trees are $r$-uniform hypertrees (see \cite[Definition 3.1]{GW2023}). 
\end{remark}

Applying Theorem \ref{thm.r-uni} in a straightforward manner immediately yields the following results.

\begin{cor} \label{pmd.tree}
    Let $H$ be an $r$-uniform good forest and denote by $\Delta(H)$ the maximal degree of a vertex in $H$. Then $\pmd(H)=\Delta(H)$.
\end{cor}

\begin{cor}
    Let $H$ be an $r$-uniform good forest on $[n]$. Then for $\Delta(H) +1 \leq d \leq \frac{1}{n}\lceil \binom{n+r-1}{n-1}\rceil -1$, every coordinate sections of the variety $S_{n,r}^d$ with respect to $H$ are irreducible.
\end{cor}

\begin{remark}
    Since hypertree (see \cite[Definition 3.1]{GW2023}) does not have any closed walk, we obtain \cite[Theorem 1.4]{GW2023}, which asserts that for an $r$-uniform hypertree $H$, one has $\pmd(H)=\Delta(H)$.
\end{remark}
 
\begin{cor}
    Let $C_{(r-1)m}$ be a loose cycle, where $r>2$ and $m\geq2$. Then $$\pmd(C_{(r-1)m})=\begin{cases}
        2, & \text{if $m$ is even},\\
        3, & \text{if $m$ is odd}.
    \end{cases}$$
\end{cor}

\begin{remark}
    Let $m>2$ and $C_m$ be a cycle graph. Then $\pmd(C_m)=3$, since even cycles have alternate closed walk as an induced subgraph. 
\end{remark}

An edge of an $r$-uniform hypergraph is said to be a \textit{pendant} if it has a vertex of degree one. Note that pendant edges do not contribute to strong alternate closed walks. The following result is analogous to \cite[Theorem 2.3]{FG2022}.

\begin{cor}  \label{jointpendant}
    Let $H$ and $H'$ are $r$-uniform hypergraphs. If $H$ can be obtained from $H'$ by adding pendent edges such that each pendant edge has $r-1$ vertices of degree $1$, then $\pmd(H)=\max\{\pmd(H'),\Delta(H)\}$.
\end{cor}

\begin{remark}
    Using Corollary \ref{jointpendant} one can derive the $\pmd$ of uni-loose-cyclic hypergraphs. These hypergraphs are identified by having only one closed walk, and this closed walk forms a loose cycle.
\end{remark}

\begin{definition} \cite[Definition 2.4]{CDG2020} \label{def.cs}
    Let $S=\mathbb{K}[x_{ij} \mid i\in [n], j \in[d]]$ be a polynomial ring with $\mathbb{Z}^n$ multigraded induced by $\deg(x_{ij})=\mathfrak{e_i} \in \mathbb{Z}^n$. Let $T=\mathbb{K}[x_{i1} \mid i\in [n]]$  be a polynomial ring with $\mathbb{Z}^n$ multigraded structure induced by that of $S$. A $\mathbb{Z}^n$-graded ideal $I$ of $S$ is said to be Cartwright-Sturmfels ideal if a radical Borel fixed ideal $J$ exists with the same multigraded Hilbert-series. An ideal $I \subset S$ is said to be Cartwright-Sturmfels$^*$ ideal if there exists a $\mathbb{Z}^n$-graded ideal $J$ of $S$ extended from $T$ such that $I$ an $J$ has the same multigraded Hilbert-series.
\end{definition}

\begin{remark} \label{rem.CSI}
    \cite[Theorem 1.16(5)]{CDG2015} The family of Cartwright-Sturmfels ideals is closed under any multigraded linear section.
\end{remark}

In the following theorem, we prove that the LSS ideal of a good forest is the Cartwright-Sturmfels ideal. 

\begin{theorem} \label{Thm.cs}
    Let $S=\mathbb{K}[x_{ij} \mid i\in [n], j \in[d]]$ be a polynomial ring with $\mathbb{Z}^n$ multigraded induced by $\deg(x_{ij})=\mathfrak{e_i} \in \mathbb{Z}^n$. Let $H=(V=[n],E)$ be a good forest. Then the ideal $L_H^\mathbb{K}(d)$ is a Cartwright-Sturmfels ideal. 
\end{theorem}
\begin{proof}
    First, we assume that the generators $f_e^{(d)}$ of $L_H^\mathbb{K}(d)$ form a regular sequence. Since for each edge $e \in E$, one adds a monomial $m_e$ to $f_e^{(d)}$ with the same degree such that $m_e$ and $m_{e'}$ are coprime if $e \neq e'$. Then, with a suitable term order, initial terms of $f_e^{(d)}$ are pairwise coprime monomials $m_e$. By Remark \ref{rem.CSI} it is enough to show that the ideal $(f_e^{(d)}+m_e \mid e\in E)$ is a Cartwright-Sturmfels ideal. Since $f_e^{(d)}$ of $L_H^\mathbb{K}(d)$ form a regular sequence the $K$-polynomial of $S/L_H^\mathbb{K}(d)$ is $F(x)=\prod_{\{i_1,\ldots,i_r\} \in E} (1-x_{i_1}\cdots x_{i_r}) \in \mathbb{Q}[x_1,\ldots,x_n]$. By Cartwright-Sturmfels$^*$ property, it is enough to show that the existence of a monomial ideal $I$ in the polynomial ring $S_1=\mathbb{K}[y_1,\ldots,y_n]$ with $\mathbb{Z}^n$ multigraded induced by $\deg(y_i)=\mathfrak{e_i} \in \mathbb{Z}^n$ such that $K$-polynomial of $I$ is $F(1-x_1,\ldots,1-x_n)$ as an $S_1$-module. We claim that $K$- polynomial of $I=\prod_{\{i_1,\ldots,i_r\} \in E}(y_{i_1},\ldots,y_{i_r})$ is $F(1-x_1,\ldots,1-x_n)$. That is, we need to prove that the tensor product $K_E = \otimes_{\{i_1,\ldots,i_r\}\in E}K_{\{i_1,\ldots,i_r\}}$ of the truncated Koszul complexes 
    $$K_{\{i_1,\ldots,i_r\}}: 0 \rightarrow S_1(-\mathfrak{e_{i_1}} \cdots -\mathfrak{e_{i_r}}) \rightarrow S_1(-\mathfrak{e_{i_1}}) \oplus \cdots \oplus S_1(-\mathfrak{e_{i_r}})$$
    associated to $y_{i_1},\ldots,y_{i_r}$ resolves the ideal $I$. Set edge set $E'=E \setminus \{\{a_1,\ldots,a_r\}\}$, and ideals  $I'= \prod_{\{i_1,\ldots,i_r\} \in E'}(y_{i_1},\ldots,y_{i_r})$ and $I''= (y_{a_1},\ldots,y_{a_r})$. Then, by induction on the number of edges, one has $K_{E'}$ resolves the ideal $I'$. Since the $H$ is a good forest, for an edge $\{a_1,\ldots,a_r\}$ one has $r-1$ variables of $y_{a_1},\ldots,y_{a_r}$ does not appear in the generators of $I'$. Then, $y_{a_1},\ldots,y_{a_r}$ forms a regular $I'$-sequence. Therefore, Tor$_{\ell}^{S_1}(I',I'')=0$ for all $\ell \geq 1$. Hence $K_E$ resolves $I' \otimes I''$ and $I' \otimes I'' = I'I''$. This proves that the ideal $L_H^\mathbb{K}(d)$ is a Cartwright-Sturmfels ideal.
\end{proof}

The following remark provides the necessary conditions for LSS-ideals of $r$-uniform hypergraphs to be complete intersections, and and the proof is analogous to \cite[Lemma 2.1]{ANC2023}. 
\begin{remark} \label{rem.nci}
    For an $r$-uniform hypergraph $H$, if $d > \Delta(H)$, then $L_H^\mathbb{K}(d)$ is not a complete intersection.
\end{remark}

In the following, we prove the radical complete intersection property of LSS-ideals of $r$-uniform good forest.

\begin{proof} [Proof of Theorem~\ref{thm.goodfor}]
     $(1)$. It follows from Theorem \ref{Thm.cs} that $L_H^{\mathbb{K}}(d)$ is a Cartwright–Sturmfels ideal. In particular, $L_H^{\mathbb{K}}(d)$ and all its initial ideals are radical for all $d$. Assertions $(2)$ and $(3)$ follows from Corollary \ref{pmd.tree}, Remark \ref{rem.nci} and Equation (\ref{eq.main}).
\end{proof}

\vspace{2mm}

\noindent {\bf Acknowledgement.} 

The authors would like to thank the anonymous referee for their helpful comments and valuable suggestions.
The first author is supported by the Scientific and Technological Research Council of Turkey T\"UB\.{I}TAK under the Grant No: 124F113. The second author is partially supported by the Anusandhan National Research Foundation (ANRF), Government of India, under the Core Research Grant (File No.: CRG/2023/007668).

\end{document}